\documentclass[12pt,twoside,reqno]{amsart}
\usepackage[colorlinks=true,citecolor=blue]{hyperref}
\usepackage{mathptmx, amsmath, amssymb, amsfonts, amsthm, mathptmx, enumerate, color}
\usepackage{graphicx}
\usepackage{epstopdf}

\newtheorem{proposition}{Proposition}[section]
\theoremstyle{definition}

\theoremstyle{plain}
\newtheorem{Thm}{Theorem}
\newtheorem{Lem}{Lemma}
\newtheorem{Def}{Definition}
\newtheorem{Cor}{Corollary}
\numberwithin{equation}{section}

\theoremstyle{definition}

\def\re{\mathbb R}
\def\na{\mathbb N}
\def\N{\mathbb N}

\def\lV{\left\Vert }
\def\rV{\right\Vert }
\def\lv{\left\vert }
\def\rv{\right\vert}
\def\al{\alpha}
\def\be{\beta}

\usepackage[english]{babel}
\usepackage{graphicx}
\usepackage{framed}
\usepackage[normalem]{ulem}
\usepackage{amsmath}
\usepackage{amsthm}
\usepackage{amssymb}
\usepackage{amsfonts}
\usepackage{enumerate}
\usepackage{multirow}
\usepackage{booktabs,siunitx,makecell}
\usepackage[referable]{threeparttablex}
\usepackage[shortlabels]{enumitem}
\usepackage[utf8]{inputenc}
\usepackage[top=1 in,bottom=1in, left=1 in, right=1 in]{geometry}
\usepackage{lettrine}
\usepackage{comment}
\usepackage{cases}
\usepackage{hyperref}
\usepackage{url}

\begin{document}
\setcounter{page}{1}

\vspace*{1.0cm}
\title[A circumcenter-reflection method for FPP]
{A circumcentered-reflection method for finding common fixed points of firmly nonexpansive operators}

\author[R. Arefidamghani, R. Behling, A.N. Iusem, L.-R. Santos]{R. Arefidamghani$^1$
R. Behling$^2$, A.N. Iusem$^{3,*}$, L.-R. Santos$^4$}
\maketitle

\begin{center}
{\footnotesize {\it
$^1$Instituto de Matem\'atica Pura e Aplicada,
Estrada Dona Castorina 110, Jardim Bot\^anico, CEP 22460-320, Rio de Janeiro, RJ,
Brazil\\
$^2$School of Applied Mathematics, Funda\c c\~ao Get\'ulio Vargas,
Rio de Janeiro, Brazil\\
$^3$Instituto de Matem\'atica Pura e Aplicada,
Estrada Dona Castorina 110, Jardim Bot\^anico, CEP 22460-320, Rio de Janeiro, RJ
Brazil\\
$^4$Department of Mathematics,
Federal University of Santa Catarina,
Blumenau, SC, Brazil}}
\end{center}

\bigskip
\medskip

\centerline{\bf Honoring Prof. Yair Censor in his 80th birthday}

\bigskip

\vskip 4mm {\small\noindent {\bf Abstract.}
The circumcentered-reflection method (CRM) has been recently proposed as a methodology for accelerating several 
algorithms for solving the Convex Feasibility Problem (CFP), equivalent to finding a common fixed-point of the
orthogonal projections onto a finite number of closed and convex sets. In this paper, we apply CRM 
to the more general Fixed Point Problem (denoted as FPP), consisting of finding
a common fixed-point of operators belonging to a larger family of operators, namely firmly nonexpansive operators.
We prove than in this setting, CRM is globally convergent to a common fixed-point (supposing at least one exists).
We also establish linear convergence of the sequence generated by CRM applied to FPP, under a not too
demanding error bound assumption, and provide an estimate
of the asymptotic constant. We provide solid numerical evidence of the superiority of CRM when compared to the
classical Parallel Projections Method (PPM). Additionally, we present certain results of convex combination of orthogonal 
projections, of some interest on its own.

\medskip


\medskip

\renewcommand{\thefootnote}{}
\footnotetext{ $^*$Corresponding author.
\par
E-mail addresses: reza.arefidamghani@immpa.br (R. Arefidamghani), rogerbehling@gmail.com (R. Behling), 
iusp@impa.br (A.N. Iusem), l.r.santos@ufsc.br (L.-R. Santos).
\par
Received ; Accepted}
\section{Introduction}\label{s1}

We start by recalling the Convex Feasibility Problem (CFP), which consists of 
finding a point in the intersection of a finite number of closed convex subsets of $\re^n$. CFP is clearly equivalent to solving a finite system
of convex inequalities in $\re^n$, and it can be also rephrased as the problem of finding a common fixed-point of the orthogonal projections
onto such subsets. A natural extension of CFP is the problem of finding a common fixed-point of a finite set of operators other than 
orthogonal projections, but sharing some of their properties. A vast literature on the subject has been developed; we cite just a few
references, namely \cite{CeS}, \cite{Mou}, \cite{YLY} and \cite{ZhH}. 
In this paper we will consider a particular generalization of orthogonal projections, namely {\it firmly nonexpansive operators}.

We define next this family of operators, together with two related families.

\begin{Def}\label{d1}
An operator $T:\re^n\to\re^n$ is said to be:
\begin{itemize}
\item[i)]  {\rm nonexpansive} when $\lV T(x)-T(y)\rV\le\lV x-y\rV$ for all $x,y\in\re^n$.
\item[ii)] {\rm nonexpansive plus} when it is nonexpansive, and whenever  
$\lV T(x)-T(y)\rV=\lV x-y\rV$ it holds that $T(x)-T(y)=x-y$.
\item[iii)] {\rm firmly nonexpansive} when
\begin{equation}\label{e1}
\lV T(x)-T(y)\rV^2\le\lV x-y\rV^2-\lV (T(x)-T(y))-(x-y)\rV^2
\end{equation} 
for all $x,y\in\re^n$.
\end{itemize}
\end{Def}

It is immediate that firmly nonexpansive operators are nonexpansive plus, and nonexpansive plus operators are nonexpansive.
It is well known and easy to prove that orthogonal projections onto closed and convex sets are firmly nonexpansive. 
The notation {\it nonexpansive plus} is not standard; we adopt it because of 
the analogy with copositive plus matrices.

Let $T_1, \dots ,T_m:\re^n\to\re^n$ be firmly nonexpansive operators. The problem of finding a common fixed-point of $T_1, \dots, T_m$
(i.e., a point $\bar x\in\re^n$ such that $T_i(\bar x)=\bar x$ for all $i\in\{1, \dots ,m\}$) will be denoted as FPP. 
The set of common fixed-points of the $T_i$'s will be denoted as Fix$(T_1, \dots ,T_m)$.  
Two classical methods for FPP are the Sequential Projection
Method (SPM) and the Parallel Projection Method (PPM), which can be traced
back to \cite{Kac}, \cite{Cim} respectively, and are defined as follows. 
Consider the operators $\widehat T,
\overline T:\re^n\to\re^n$ given by $\widehat T=T_m\circ\dots\circ T_1$,
$\overline T=\frac{1}{m}\sum_{i=1}^m T_i$. Starting from an arbitrary $x^0\in\re^n$,
SPM and PPM generate  sequences $\{x^k\}$ given by $x^{k+1}=\widehat T(x^k)$,
$x^{k+1}=\overline T(x^k)$ respectively. When Fix$(T_1, \dots T_m)\ne\emptyset$
the sequences generated by both
methods are known to be globally convergent to points belonging to a point in Fix$(T_1, \dots, T_m)$,
i.e., to solve FPP.  
See \cite{CeZ} for an in-depth study of these and other 
projections methods for FPP. 

 An interesting relation between SPM and PPM was
found in \cite{Pie}. Given firmly nonexpansive operators $T_1, \dots, T_m:\re^n\to\re^n$, define the operator
$\widetilde T:\re^{nm}\to\re^{nm}$ as $\widetilde T(x^1, \dots ,x^m)=(T_1(x^1), \dots ,T_M(x^m))$,
with $x^i\in\re^n$ ($1\le i\le m)$. It is rather immediate to check that $\widetilde T$ is firmly nonexpansive. Consider the set
$\widetilde U=\{(x,\dots,x): x\in\re^n\}\subset\re^{nm}$, and let $P_{\widetilde U}:\re^{nm}\to\widetilde U$
be the orthogonal projection onto $\widetilde U$.   
Define $\{\bar x^k\}\subset\re^{nm}$ as the sequence resulting from applying SPM, as defined above, 
to the operators $\widetilde T, P_{\widetilde U}$, starting from a point $\bar x^0=(x^0, \cdots ,x^0)\in\widetilde U$,
i.e., take $\bar x^{k+1}=P_{\widetilde U}(\widetilde T(\bar x^k))$. Clearly, $\bar x^k$ belongs to $\widetilde U$
for all $k$, so that we may write $\bar x^k=(x^k,\dots,x^k)$ with $x^k\in\re^n$. It was proved in 
\cite{Pie} that $x^{k+1}=\overline T(x^k)$, i.e., a step of SPM applied to two specific firmly nonexpansive operators in
the product space $\re^{nm}$ is equivalent to a step of PPM in the original space $\re^n$.
Thus, SPM with just two operators plays a sort of special role, and deserves a name of its own. We
will call it the {\it Method of Alternating Projections} (MAP from now on). 
Observe that in the equivalence above one of the two sets in the product space, 
namely $\widetilde U$, is a linear subspace.
This fact will be essential for the convergence of the Circumcentered-Reflection Method (CRM from now on),
applied for solving FPP. 

We reckon that the use of the word ``projections" in the names of SPM, PPM and MAP applied to FPP is an abuse of notation,
since in general there are no projections involved in FPP. Indeed, they correspond to these methods applied to CFP, a particular case of FPP.
We keep them because the structure of the methods applied to either CFP and FPP is basically the same.  

We proceed to describe CRM.
Take three non-collinear points $x,y,z\in\re^n$, and let $M$ be their affine hull. The {\it circumcenter} circ$(x,y,z)$ 
is the center of the circle in $M$
passing through $x,y,z$ (or, equivalently, the point in $M$ equidistant from $x,y,z$). 
It is easy to check that circ$(x,y,z)$ is well defined. Now we take two firmly nonexpansive operators $A,B:\re^n\to\re^n$ and define $Q=A\circ B$.
Under adequate assumptions, the sequence $\{x^k\}\subset\re^n$ defined by 
\begin{equation}\label{e0}
x^{k+1}=Q(x^k)=A(B(x^k))
\end{equation}
is expected
to converge to a common fixed-point of $A$ and $B$. Note that, if $A,B$ are orthogonal projections onto
convex sets $K_1, K_2$, then MAP turns out to be a special case of this iteration, and
Fix$(A,B)=K_1\cap K_2$. CRM can be seen as an acceleration technique 
for the sequence defined by \eqref{e0}. Define the reflection operators $A^R, B^R:\re^n\to\re^n$ as $A^R=2A-I, B^R=2B-I$,
where $I$ stands for the identity operator in $\re^n$. The CRM operator $C:\re^n\to\re^n$ is defined
as 
$C(x)= $circ$(x, B^R(x), A^R(B^R(x)))$,
i.e., the circumcenter of the points $x, B^R(x),A^R(B^R(x))$. The CRM sequence $\{x^k\}\subset\re^n$,
starting at some $x^0\in\re^n$, is then defined as
$x^{k+1}=C(x^k)$.

CRM was introduced in \cite{BBS1}, \cite{BBS2} and has been successfully applied for accelerating several methods for solving CFP, like MAP, PPM and
the Douglas-Rachford Method (DRM), outperforming all of them. It was further enhanced in \cite{AABBIS},
\cite{ABBIS}, 
\cite{BOW1}, \cite{BOW2}, \cite{BOW3}, \cite{BOW4}, \cite{BOW5},\cite{BBS3}, \cite{BBS4}, \cite{DHL1}, 
\cite{DHL2} and \cite{Ouy}.  CRM was shown in \cite{BBS2} to converge to a solution of CFP. In \cite{ABBIS} it was proved that, 
under a not too demanding error bound condition, the sequences generated by
MAP and CRM for solving CFP converge linearly, but the asymptotic constant for CRM is better than the one for MAP. 
This superiority was widely confirmed in the numerical experiences exhibited in \cite{AABBIS}. 

Here, we will apply CRM for solving FPP with firmly nonexpansive operators $T_1, \dots T_m:\re^n\to\re^n$ in the following way.
We will apply it to two operators in $\re^{nm}$, namely $\widetilde T$ and $P_{\widetilde U}$ as defined above, 
starting from a point in $\widetilde U$.
Note that, since $\widetilde U$ is a linear subspace, the operator $P_{\widetilde U}$ is affine.

The main purpose of this paper consists of establishing that CRM, when applied to FPP, is globally convergent, 
that linear convergence is achieved by both CRM and MAP under an error bound condition, and that 
CRM is computationally much faster than MAP, as corroborated by solid numerical evidence. 
We were not able to prove the
superiority of CRM in terms of the asymptotic constant of linear convergence, but our numerical experiments
suggest that a theoretical superiority is likely to hold. This issue is left as a subject for future
research. 

The paper is organized as follows. In Section \ref{s2} we present certain results, of some interest on its own,  on convex combinations 
of orthogonal projections, which we take as a prototypical family of firmly nonexpansive operators (beyond orthogonal projections themselves).
In Section \ref{s3} we prove global convergence of CRM applied for solving FPP.  We prove in Section \ref{s4}  that, under a reasonable error bound assumption, convergence of CRM applied for solving FPP is linear, and we provide as well an estimate of the asymptotic constant,  which holds also for MAP. In Section \ref{s5}
we present our numerical experiments which show that CRM categorically outperforms PPM. In these experiments, 
we use the family of firmly nonexpansive operators studied in Section \ref{s2}.   

\section{Some properties of firmly nonexpansive operators}\label{s2}

We start with some elementary properties of nonexpansive plus and firmly nonexpansive operators (see Definition \ref{d1}).    
		
\begin{proposition}\label{p1}
\begin{itemize} 
\item[i)] Compositions of nonexpansive plus operators are nonexpansive plus.
\item[ii)] Convex combinations of firmly nonexpansive operators are firmly nonexpansive.
\end{itemize}
\end{proposition}

\begin{proof}
\begin{itemize} 
\item[i)] Suppose that $S,T$ are nonexpansive plus operators. Then
\begin{equation}\label{e2}
\lV S(T(x))-S(T(y))\rV\le\lV T(x)-T(y)\rV\le\lV x-y\rV,
\end{equation}
by nonexpansiveness of $S,T$,
and if $\lV S(T(x))-S(T(y))\rV=\lV x-y\rV$, then equality holds throughout \eqref{e2},
so that, using the ``plus" property of $S,T$, we have
$S(T(x))-S(T(y))=T(x)-T(y)=x-y$, establishing the result.
\item[ii)] Take firmly nonexpansive operators $T_1\dots, T_m$ and nonnegative scalars $\al_1,\dots ,\al_m$ such that
$\sum_{i=1}^m\al_i=1$. Let $\overline T=\sum_{i=1}^m\al_iT_i$. We prove next that $\overline T$ is firmly nonexpansive.

Note that \eqref{e1} is equivalent to
\begin{equation}\label{e3}
\lV T(x)-T(y)\rV^2\le\langle T(x)-T(y), x-y\rangle.
\end{equation}
It suffices to check that $\overline T$ satisfies \eqref{e3}, and we proceed to do so.
$$
\lV\overline T(x)-\overline T(y)\rV^2=\lV\sum_{i=1}^m\al_i(T_i(x)-T_i(y))\rV^2\le
\sum_{i=1}^m\al_i\lV T_i(x)-T_i(y)\rV^2
$$
$$
\le\sum_{i=1}^m\al_i\langle T_i(x)-T_i(y),x-y\rangle
=\Big\langle\sum_{i=1}^m\al_i\left(T_i(x)-T_i(y)\right),x-y\Big\rangle=\langle\overline T(x)-\overline T(y),x-y\rangle,
$$
using the convexity
of $\lV\cdot\rV^2$ in the first inequality and the fact that the $T_i$'s satisfy \eqref{e3} 
in the second one.
\end{itemize}
\end{proof}

For an operator $T:\re^n\to\re^n$, we denote as $F(T)$
the set of its fixed points, i.e., $F(T)=\{x\in\re^n: T(x)=x\}$ (we comment that Fix$(\cdot, \cdot)$ denotes 
the set of common fixed points of two or more operators).
We will also need the following ``acute angle" property of firmly nonexpasive operators.

\begin{proposition}\label{pp1} 
Let $T:\re^n\to\re^n$ be a firmly nonexpansive operator. Then $0\ge\langle T(x)-y,T(x)-x\rangle$ for all $x\in\re^n$ and all
$y\in F(T)$.
\end{proposition}
\begin{proof}
Immediate from \eqref{e1}.
\end{proof}

We continue by stating, for future reference, some elementary and well known properties of orthogonal projections onto closed and convex sets.

Let $\subset\re^n$ be closed and convex. The {\it orthogonal projection} $P_C:\re^n\to C$ is defined
as $P_C(x)={\rm argmin}_{y\in C} \lV x-y\rV$. 

\begin{proposition}\label{p2} 
If $C\subset\re^n$ is closed and convex, then
\begin{itemize}
\item[i)] $z=P_C(x)$ if and only if $\langle x-z,y-z\rangle\le 0$ for all $x\in\re^n$ and all $y\in C$.
\item[ii)] $P_C$ is firmly nonexpansive.
\item[iii)] $F\left(P_C\right)=C$.
\item[iv)] Take $x\in\re^n$ and let $z=P_C(x)$. Then, $P_C(z+\al (x-z))=P_C(x)$ for all $\al\ge 0$.
\item[v)] Define $h:\re^n\to\re$ as $h(x)=\lV x-P_C(x)\rV^2$. Then $h$ is continuously differentiable and
$\nabla h(x)=2\left(x-P_C(x)\right)$.
\end{itemize}
\end{proposition}

\begin{proof} 
Elementary.
\end{proof}

It is worthwhile to comment at this point that the composition of two firmly nonexpansive operators 
may fail to be firmly nonexpansive: consider $A=\{(x_1,x_2)\in\re^2: x_2=0\}$, $B=\{(x_1,x_2)\in\re^2: x_2=x_1\}$.
$P_A$ and $P_B$ are firmly nonexpansive by Proposition \ref{p2}(ii), but its composition $P_A\circ P_B$ 
fails to satisfy
\eqref{e3} with $x=(0,0)$ and $y=(2,-1)$.

We present next some properties of the set of fixed points of combinations of orthogonal projections.
They have been proved, e.g., in \cite{DeI}, \cite{IuD}, but we include the proofs for the sake of completeness. From now on, for $C\subset\re^n$ and $x\in\re^n$, dist$(x,C)$ will denote the Euclidean distance between $x$ and $C$.

\begin{proposition}\label{p3} 
Consider closed and convex sets $C_1\dots, C_m\subset\re^n$ and nonnegative scalars $\al_1,\dots ,\al_m$ such that
$\sum_{i=1}^m\al_i=1$. Denote $P_i=P_{C_i}$ and let $\overline P=\sum_{i=1}^m\al_iP_i$. Define $g:\re^n\to\re$ as
$g(x)=\sum_{i=1}^m\al_i\lV x-P_i(x)\rV^2=\sum_{i=1}^m\al_i{\rm dist}(x,C_i)^2$ 
and let $C=\cap_{i=1}^mC_i$. Then,
\begin{itemize} 
\item[i)] $F(\overline P) =\{x\in\re^n:\nabla g(x)=0\}$, i.e., since $g$ is convex, 
the set of fixed points of $\overline P$ (if nonempty) is precisely the set of
minimizers of $g$.
\item[ii)] If $C\ne\emptyset$, then $F(\overline P)=C$.
\end{itemize}
\end{proposition} 

\begin{proof} 
\begin{itemize}
\item[i)] By Proposition \ref{p2}(v), 
$$
\nabla g(x)=2\sum_{i=1}^m\al_i(x-P_i(x))=2\left(x-\sum_{i=1}^m\al_iP_i(x)\right)=2(x-\overline P(x)),
$$
so that $\nabla g(x)=0$ iff $x=\overline P(x)$ iff $x\in F(\overline P)$.
\item[ii)] Clearly, $C\subset F(\overline P)$. For the converse inclusion note that
when $C\ne\emptyset$, we have $g(x)=0$ for all $x\in C$, so that the minimum value of $g$ is indeed $0$,
and the set of minimizers of $g$ coincides with the set of its zeroes, which is $C$, because $g(x)>0$ whenever
$x\notin C$. The result follows then from item (i).
\end{itemize}
\end{proof}

The next result provides a more accurate description of the set $F(\overline P)$ when $m=2$,
i.e., for the case of a convex combination of the orthogonal projections onto two closed and convex sets.

Let $A,B\subset\re^n$ be two closed sets. Take $\al\in (0,1), \overline P=(1-\al)P_A+\al P_B$.
Define $D\subset A\times B$ as $D=\{(x,y)\in A\times D: \lV x-y\rV={\rm dist}(A,B)\}$.  $S_A, S_B$ will 
denote the projections of $D$ onto $A,B$ respectively, i.e., $S_A=\{x\in A: \exists y\in B\,\,{\rm with}\,\,
(x,y)\in D\}$, $S_B=\{y\in B: \exists x\in A \,\, {\rm with} \,\, (x,y)\in D\}$. In other words, $D$ consist of the pairs
in $A\times B$ which realize the distance between $A$ and $B$, $S_A$ is the set of points in $A$ which realize
the distance to $B$, and $S_B$ is the set of points in $B$ which realize the distance to $A$. We remark that $D$ may be empty;
take for instance $A=\{(x_1,x_2)\in\re^2: x_2\le 0\}$, $B=\{(x_1,x_2)\in\re^2: x_2\ge e^{x_1}\}$.

\begin{proposition}\label{p4} 
With the notation above,
\begin{itemize}
\item[i)] For all $(x,y), (x,',y')\in D$, it holds that $x-y=x'-y'$.
\item[ii)] Take $(x,y)\in D, \al\in (0,1)$ and define $w=(1-\al)x+\al y$. Then $P_A(w)=x, P_B(w)=y$.
\item[iii)] $F(\overline P)=\{w=(1-\al)x+\al y: (x,y)\in D\}$.
\end {itemize}
\end{proposition}

\begin{proof}
\begin{itemize}
\item[i)] Since, for any $(x,y)\in D$ the pair $(x,y)$ realizes the distance between $A$ and $B$, it follows
that $P_B(x)=y, P_A(y)=x$ for all $(x,y)\in D$, and hence $P_A(P_B(x))=x$ for all $x\in S_A$. So, for all
$(x,y), (x',y') \in D$, we have
\begin{equation}\label{e4}
\lV x-x'\rV=\lV P_A(P_B(x))-P_A(P_B(x'))\rV\le\lV P_B(x)-P_B(x')\rV\le\lV x-x'\rV,
\end{equation} 
using Proposition \ref{p2}(ii). It follows that equality holds throughout \ref{e4}, and since $P_A\circ P_B$
is nonexpansive plus by Proposition \ref{p1}(i), because both $P_A$ and $P_B$ are firmly nonexpansive (and
so nonexpansive plus) by 
Proposition \ref{p2}(ii), we conclude from Definition \ref{d1}(ii) that
$x-x'=P_B(x)-P_B(x')=y-y'$ which implies that $x-y=x'-y'$.
\item[ii)] Take $(x,y)\in D$, so that $x\in A$. Then $w=y+(1-\al)(x-y)=P_B(x)+(1-\al)\left(x-P_B(x)\right)$.
Since $1-\al>0$, it follows from Proposition \ref{p3}(iv) that $P_B(w)=y$. A similar argument 
establishes that $P_A(w)=x$.
\item[iii)] Take $w=(1-\al)x+\al y$ with $(x,y)\in D$. Then, by (ii),
$w=(1-\al)P_A(w)+\al P_B(w)=\overline P(w)$, and hence $w\in F(\overline P)$, so that 
 $\{w=(1-\al)x+\al y: (x,y)\in D\}\subset F(\overline P)$. For the converse inclusion, consider any $x\in F(\overline P)$,
i.e., 
\begin{equation}\label{e5}
x=(1-\al)P_A(x)+\al P_B(x).
\end{equation}
Let $\delta=$ dist$(A,B), \eta=\lV P_A(x)-P_B(x)\rV$. It suffices to check that 
$\left(P_A(x),P_B(x)\right)\in D$, i.e., that 
\begin{equation}\label{e6}
\eta=\delta.
\end{equation} 
From \eqref{e5}, we get
$$
\lV x-P_A(x)\rV=\al\lV P_B(x)-P_A(x)\rV=\al\eta,
$$
$$
\lV x-P_B(x)\rV=(1-\al)\lV P_B(x)-P_A(x)\rV=(1-\al)\eta,
$$
implying that 
\begin{equation}\label{e7}
g(x)=(1-\al)\lV x-P_A(x)\rV^2+\al\lV x-P_B(x)\rV^2=[(1-\al)\al^2+\al(1-\al)^2]\eta^2=(1-\al)\al\eta^2.
\end{equation}
Take now any pair $(u,v)\in D$, so that $\lV u-v\rV=\delta$, and let $w=(1-\al)u+\al v$. 
By item(ii), $u=P_A(w), v=P_B(w)$, so that
$$
\lV w-P_A(w)\rV=\al\lV P_B(w)-P_A(w)\rV=\al\lV u-v\rV=\al\delta,
$$
$$
\lV w-P_B(w)\rV=(1-\al)\lV P_B(w)-P_A(w)\rV=(1-\al)\lV u-v\rV=(1-\al)\al\delta,
$$
and hence,
\begin{equation}\label{e8}
g(w)=(1-\al)\lV w-P_A(w)\rV^2+\al\lV w-P_B(w)\rV^2=[(1-\al)\al^2+\al(1-\al)^2]\delta^2=(1-\al)\al\delta^2.
\end{equation}
By Proposition \ref{p3}(i), $x$ is a minimizer of $g$, so that $g(x)\le g(w)$, which implies, in view
of \eqref{e7},\eqref{e8}, and the fact that $\al\in (0,1)$, that $\eta\le\delta$. On the other hand,
$\eta=\lV P_A(x)-P_B(x)\rV$ with $P_A(x)\in A, P_B(x)\in B$, so that $\eta\ge$ dist$(A,B)=\delta$. We conclude that
\eqref{e6} holds, and the result is established.
\end{itemize}
\end{proof}

We deal now with the main result of this section, which we describe next. The prototypical examples of 
firmly nonexpansive operators are the orthogonal projections onto closed and convex sets. Proposition \ref{p1}(ii)
provides a larger class of firmly nonexpansive operators, namely convex combinations of orthogonal projections.
It is therefore relevant to check that the second class is indeed larger, i.e., that, generically, 
convex combinations of orthogonal projections are not orthogonal projections themselves. We will prove that
this is indeed the case when the intersection of the convex sets is nonempty. However, when this intersection
is empty, a convex combination of orthogonal projections may be itself an orthogonal projection. 
We will establish a necessary and sufficient 
condition for this situation to occur, for the case of two convex sets.  

\begin{proposition}\label{p5}
Consider closed and convex sets $C_1\dots, C_m\subset\re^n$ and nonnegative scalars $\al_1,\dots ,\al_m$ such that
$\sum_{i=1}^m\al_i=1$. Denote $C=\cap_{i=1}^m C_i$, $P_i=P_{C_i}$ and let $\overline P=\sum_{i=1}^m\al_iP_i$.
Assume that $C\ne\emptyset$. If there exists $E\subset\re^n$ such that $\overline P=P_E$ then 
$E=C_1= \dots = C_m$. 
\end{proposition}

\begin{proof}
By Propositions \ref{p3}(ii) and \ref{p2}(iii),
\begin{equation}\label{e9}
C=F(\overline P)=F(P_E)=E.
\end{equation}
Take $x\in C_i$. Let $\ell={\rm argmax}_{1\le j\le m}\{\lV x-P_j(x)\rV\}$, $w=\sum_{j=1}^m\al_jP_j(x)=\overline P(x)=P_E(x)$,
so that $w\in $ Im$(P_E)=E=C$, using \eqref{e9}, and hence $w\in C_\ell$. It follows that
$$
\lV x-P_\ell(x)\rV\le\lV x-w\rV=\lV\sum_{i=j}^m\al_j(x-P_j(x))\rV\le\sum_{j=1}^m\al_j\lV x-P_j(x)\rV
$$
$$
=\sum_{j=1, j\ne i}^m\al_j\lV x-P_j(x)\rV\le
\sum_{j=1, j\ne i}^m\al_j\lV x-P_\ell(x)\rV=
$$
\begin{equation}\label{e10}
\left(\sum_{j=1, j\ne i}^m\al_j\right)\lV x-P_\ell(x)\rV=(1-\al_i)\lV x-P_\ell(x)\rV,
\end{equation}
using the convexity of $\lV\cdot\rV$ in the first inequality, the fact that $x\in C_i$ in the second 
equality and the definition of $\ell$ in the second inequality.
It follows from \eqref{e10} that $\al_i\lV x-P_\ell(x)\rV\le 0$, so that $\lV x-P_\ell(x)\rV =0$.
Since $0\le\lV x-P_j(x)\rV\le\lV x-P_\ell(x)\rV$ for all $j$ by definition of $\ell$, we conclude that
$\lV x-P_j(x)\rV=0$ for all $j$, i.e. $x\in C_j$. Since $x$ is an arbitrary point in $C_i$, we get that
that $C_i\subset C_j$ for all $i,j$, i.e., $C_1= \dots =C_m$, and the result follows immediately from \eqref{e9}.
\end{proof}   

Next, we fully characterize the situation for the case of $2$ convex sets. For $A\subset\re^n$, we denote the affine
hull of $A$ as aff$(A)$.

\begin{proposition}\label{p6}
Take closed and convex sets $A,B\subset\re^n$ and $\al\in (0,1)$. Define $\overline P=(1-\al)P_A+\al P_B$.
Then, there exists a nonempty, closed and convex set $E\subset\re^n$
such that $\overline P=P_E$ if and only if there exists $c\in {\rm aff}(A)^\perp$ such that $B=A+c$.
\end{proposition}

\begin{proof}
We start with the ``only if" statement. We claim that the result holds with $E=A+\al c$. First we prove that
$P_B(x)=P_A(x)+c$ for all $x\in\re^n$. Let $z=P_A(x)+c$. By Proposition \ref{p2}(i), it suffices to prove
that $\langle x-z, y-z\rangle\le 0$ for all $x\in\re^n$ and all $y\in B=A+c$, i.e., that for all $y\in A$
we have
$$
0\ge\langle x-z,y+c-z\rangle=\langle x-P_A(x)-c, y+c-P_A(x)-c\rangle
$$
\begin{equation}\label{e11}
=\langle x-P_A(x), y-P_A(x)\rangle -
\langle c,y-P_A(x)\rangle=\langle x-P_A(x),y-P_A(x)\rangle,
\end{equation}
using in the last equality the facts that $c\in$ aff$(A)^\perp$ and $y,P_A(x)\in A$, so that $y-P(A)\in$ aff$(A)$,
and hence $\langle c,y-P_A(x)\rangle =0$. Note that  $0\le\langle x-P_A(x),y-P_A(x)\rangle$ by Proposition \ref{p2}(i),
so that the inequality in \eqref{e11} holds, and hence we have proved that $P_B(x)=P_A(x)+c$ for all $x\in\re^n$.
It follows that $\overline P=(1-\al)P_A+\al P_B=(1-\al)P_A+\al P_A+\al c=P_A+\al c$.

Now, the same argument used to prove that $P_{A+c}=P_A+c$, allow us to conclude that $P_A+\al c=P_{A+\al c}$,
so that $(1-\al)P_A+\al P_B=P_E$ with $E=A+\al c$.

Now we prove the ``if" statement. First we must identify the appropriate vector $c$. By assumption, $\overline P=P_E$,
so that $F(\overline P)=E\ne\emptyset$ by Proposition \ref{p2}(iii). It follows that $D$, as defined in Proposition \ref{p4},
is nonempty. We take any pair $(u,v)\in D$ and take $c=v-u$. By Proposition \ref{p4}(i), $c$ does not depend on the
chosen pair $(u,v)$. We must prove that $B=A+c$, and we first claim that 
\begin{equation}\label{e12}
S_B=S_A+c, 
\end{equation}
with $S_A, S_B$ 
as in Proposition \ref{p4}. Take $u\in S_B$, so that there exists $v\in S_A$ such that $(u,v)\in D$ and hence
$v=u+(v-u)=x+c$, showing that $v\in S_A+c$, and therefore $S_B\subset S_A+c$. Reversing the roles of $A,B$ we get the
reverse inclusion, and then \eqref{e12} holds. 

We show next that the assumption $\overline P=P_E$ implies that
$A=S_A, B=S_B$. Take any $x\in A$. We must prove that $x$ realizes the distance to $B$. Let 
$z=\overline P(x)=(1-\al)P_A(x)+\al P_B(x)=(1-\al)x+\al P_B(x)$. It follows from Proposition \ref{p2}(iv) that
$P_B(z)=P_B(x)$. Note that
\begin{equation}\label{e13}
(1-\al)(x-P_B(x))=z-P_B(x)=z-P_B(z). 
\end{equation}
Now $z=\overline P(x)=P_E(x)$, so that
$z\in E=F(P_E)=F(\overline P)$. By Proposition \ref{p4}(ii) and (iii), $z=(1-\al)P_A(z)+\al P_B(z)$, with
$(P_A(z),P_B(z))\in D$. It follows that 
\begin{equation}\label{e14}
z-P_B(z)=(1-\al)(P_A(z)-P_B(z)).
\end{equation}
Since $\al \in (0,1)$, we conclude from \eqref{e13}, \eqref{e14} that $x-P_B(x)=P_A(z)-P_B(z)$,
so that, in view of the fact that $(P_A(z),P_B(z))\in D$, 
$$
{\rm dist}(x,B)=\lV x-P_B(x)\rV=\lV P_A(z)-P_B(z)\rV={\rm dist}(A,B).
$$
We have proved that $x$ realizes the distance between $A$ and $B$, i.e., that $x\in S_A$. 
Since $x$ is an arbitrary point in $A$, we have $A\subset S_A\subset A$, so that $A=S_A$.      
By the same token, $B=S_B$. In view of \eqref{e12}, we have that $B=A+c$. 

It only remains
to be verified that $c\in$ aff$(A)^\perp$. Let relint$(A)$ be the relative interior of $A$
(i.e., the interior of $A$ with respect to aff$(A)$). Take any $x\in$ relint$(A)$ and any $z\in$ aff$(A)$.
Since $x\in$ relint$(A)$, there exists $\varepsilon >0$ such that both $x+\varepsilon(z-x)$ 
and $x-\varepsilon(z-x)$ belong to $A$. Since $x\in A=S_A$, we obtain from Proposition \ref{p4}(ii) that
$x=P_A(v)$ for some $v\in S_B$, and $c=v-P_B(v)=v-x$, so that, by Proposition \ref{p2}(i),
$\langle c,y-x\rangle=\langle v-P_B(v),y-P_B(v)\rangle\le 0$ for all $y\in A$. Taking first
$y=x+\varepsilon (z-x)$ and then $y=x-\varepsilon(z-x)$, we conclude that $\varepsilon\langle c,z-x\rangle\le 0$,
$-\varepsilon\langle c,z-x\rangle\le 0$, implying that $\langle c,z-y\rangle =0$ for all $z\in$ aff$(A)$,
and hence $c\in$ aff($A)^\perp$, completing the proof.
\end{proof}

\begin{Cor}\label{c1} 
Assume that any of the equivalent statements in Proposition \ref{p6} hold
and that $A\ne B$. Then $A$ has empty interior and $A\cap B=\emptyset$.
\end{Cor}

\begin{proof}
Since $B=A+c$ and $A\ne B$, we have $c\ne 0$. Since $c\in$ aff$(A)^\perp$, we obtain that aff$(A)\ne\re^n$,
i.e. aff$(A)$ is not full dimensional and hence $A$ has empty interior. 

For the second statement, assume that $A\cap B\ne\emptyset$ and take
$x\in A\cap B$. Since $B=A+c$, we have $x=x'+c$ with $x'\in A$, so that $\lV c\rV^2=\langle c,x-x'\rangle =0$, because
$c\in $aff$(A)^\perp$ and $x,x'\in A$, so that $x-x'\in$ aff$(A)$. It follows that $c=0$, and the 
resulting contradiction entails the result.
\end{proof}

We mention that the second statement of the corollary follows also from Proposition \ref{p5}.

The ``only if" statement of Proposition \ref{p6} can be easily generalized to the case of $m$ convex sets;
unfortunately we do not have at this point a proof for the much more interesting generalization of the ``if" statement.
The following corollary contains the generalization of the ``only if" statement.

\begin{Cor}\label{c2}
Consider closed and convex sets $C_1\dots, C_m\subset\re^n$ and nonnegative scalars $\al_1,\dots ,\al_m$ such that
$\sum_{i=1}^m\al_i=1$. Denote $P_i=P_{C_i}$ and let $\overline P=\sum_{i=1}^m\al_iP_i$. Take 
$\be_2, \dots ,\be_m\in\re, c\in$ aff$(C_1)^\perp$, and assume that $C_i=C_1+\beta_i c$ for $i=2,\cdots ,m$.
Define $\bar\be=\sum_{i=2}^m\al_i\be_i, E= C_1+\bar\be c$. Then $\overline P=P_E$.
\end{Cor}

\begin{proof} The argument used in the proof of Proposition \ref{p6} shows that $P_i(x)=P_1(x)+\beta_i c$ for
$i=2, \dots ,m$, and all $x\in\re^n$, so that $\overline P(x)=P_1(x)+\bar\be c$ for all $x\in\re^n$. 
The same argument then shows that $P_1+\bar\be c=P_E$.
\end{proof}  

\section{Convergence of CRM applied to FPP}\label{s3}

In this section, we establish convergence of CRM applied to finding a point in Fix$(T,P_U)$, where $T:\re^n\to\re^n$ is firmly nonexpansive and
$P_U:\re^n\to\re^n$ is the orthogonal projection onto an affine manifold $U\subset\re^n$. As explained in Section \ref{s1}, through Pierra's formalism
in the product space $\re^{nm}$, this result entails convergence of CRM applied to finding a point in 
Fix$(T_1, \dots ,T_m)$, where $T_i:\re^n\to\re^n$ is firmly nonexpansive for $1\le i\le m$.

Our convergence analysis for CRM requires comparing the CRM and the MAP sequences, so that we start by proving convergence of the
second one, defined as 
\begin{equation}\label{ee14}
z^{k+1}=P_U(T(z^k)),
\end{equation}
starting at some $z^0\in\re^n$. This is a classical result, but we include it for the sake of self-containment.
We start with the following intermediate result.

\begin{proposition}\label{p7}
For all $x\in\re^n$ and all $y\in {\rm Fix}(T,P_U)$ it holds that
\begin{equation}\label{e15}
\lV P_U(T(x))-y\rV^2\le\lV x-y\rV^2-\lV P_U(x)-x\rV^2-\lV P_U(T(x))-T(x)\rV^2.
\end{equation}
\end{proposition}
 
 \begin{proof}
By firm nonexpansiveness of $P_U$, we have 
\begin{equation}\label{aab11}
\lV P_U(x)-y\rV^2\le\lV x-y\rV^2-\lV P_U(x)-x\rV^2
\end{equation} 
for all $x\in\re^n$, using the fact that $u\in U$.
Substituting $T(x)$ for $x$ in \eqref{aab11}, we obtain
\begin{equation}\label{et16}
\lV P_U(T(x))-y\rV^2\le\lV T(x)-y\rV^2-\lV P_U(T(x))-T(x)\rV^2 .
\end{equation}
Since $T$ is firmly nonexpansive, 
\begin{equation}\label{et162}
\lV T(x)-y\rV^2\le\lV x-y\rV^2-\lV T(x)-x\rV^2. 
\end{equation}
Now combining \eqref{et16} with \eqref{et162}, we get 
\begin{equation}\label{etr}
 \lV P_U(T(x))-y\rV^2\le\lV x-y\rV^2-\lV T(x)-x\rV^2-\lV P_U(T(x))-T(x)\rV^2 
\end{equation}
which implies the result.
\end{proof}

Using Proposition \ref{p7} we get convergence of $\{z^k\}$ using the classical argument for MAP applied to CFP,
as we show next:

\begin{proposition}\label{p8} 
If ${\rm Fix}(T,P_U)\ne\emptyset$, then the sequence $\{z_k\}$ defined by \eqref{ee14} converges 
to a point $\bar z\in {\rm Fix}(T,P_U)$.
\end{proposition}
     
\begin{proof} Take any $y\in$ Fix$(T,P_U)$. By \eqref{ee14}, $z^{k+1}=PU(T(z^k))$. Using \eqref{aab11}, we get
\begin{equation}\label{et17}
\lV z^{k+1}-y\rV^2\le\lV z^k-y\rV^2-\lV P_U(T(z^k))-T(z^k)\rV^2-\lV T(z^k)-z^k\rV^2  
\le\lV z^k-y\rV^2.
\end{equation}

It follows from \eqref{et17} that $\lV z^{k+1}-y\rV^2\le\lV z^k-y\rV$ for all $k\in\N$, so that $\{z^k\}$ 
is bounded and $\{\lV z^k-y\rV\}$ is nonincreasing
and nonnegative, therefore convergent.

Hence, rewriting \eqref{et17} as
$$
\lV P_U(T(z^k))-T(z^k)\rV^2+\lV T(z^k)-z^k\rV^2\le\lV z^k-y\rV^2-\lV z^{k+1}-y\rV^2,
$$
we conclude that 
\begin{equation}\label{ata18}
\lim_{k\to\infty}\lV T(z^k)-z^k\rV=0.
\end{equation}
Let $\bar z$ be a cluster point of the bounded sequence $\{z^k\}$.  Taking limits in \eqref{ata18} along a subsequence 
converging to $\bar z$, and using the continuity of $T$, resulting from its nonexpansiveness, we get that 
$T(\bar z)=\bar z$. Since $z^k\in U$ for all $k\in\N$ by \eqref{ee14}, we have that $\bar z\in U$, so that $\bar z\in$ Fix$(T,P_U)$. 
Taking now $y=\bar z$ in \eqref{et17},
we conclude that $\{\lV z^k-\bar z\rV\}$ is convergent, and since a subsequence of this sequence
converges to $0$, the whole sequence $\{\lV z^k-\bar z\rV\}$ converges
to $0$, i.e., $\lim_{k\to\infty}z^k=\bar z\in$ Fix$(T,P_U)$.
\end{proof}

Now we proceed to the convergence analysis of CRM applied to FPP.
Let $T:\re^n\to\re^n$ be a firmly nonexpansive operator, $U\subset\re^n$ an affine manifold, and $P_U:\re^n\to\re^n$ the orthogonal projection onto $U$.
We assume that Fix$(T,P_U)\ne\emptyset$.
We denote as $R, R_U$ the reflection operators related to $T,P_U$ respectively, i.e., 
$R(x)=2T(x)-x, R_U(x)=2P_U(x)-x$.
We define $C:\re^n\to\re^n$ as the the CRM operator, i.e., $C(z)=$ circ$\{z,R(z),R_U(R(z))\}$, where ``circ" denotes the circumcenter of three
points, as defined in Section \ref{s1}. We also define $S:\re^n\to\re^n$ as $S(x)=P_U(T(x))$, so that $S$ can be seen the MAP operator. 

 We will prove that, starting from any initial point $x^0\in U,$ the sequence $\{ x^k\}$ generated by CRM, defined as
$x^{k+1}=C(x^k)$, converges to a point in Fix$(T,P_U)$.

Our convergence analysis is close to the one in \cite{ABBIS} for CRM applied to CFP, but with several differences, resulting from the fact that now 
$T$ is an arbitrary firmly nonexpansive operator, rather than the orthogonal projection onto a convex set. One of differences is the use of the next 
property of circumcenters, which will substitute for a specific property of orthogonal projections.

\begin{proposition}\label{p9}
For all $x\in\re^n$, $\langle x-T(x),C(x)-T(x)\rangle =0$.
\end{proposition}

\begin{proof}
By the definition of the reflection, for all $x\in\re^n$, 
\begin{equation}\label{e18}
T(x)=\frac{1}{2}(R(x)+x).
\end{equation}  
By the definition of circumcenter, for all $x\in\re^n$, 
\begin{equation}\label{e19}
\lV C(x)-x\rV^2=\lV C(x)-R(x)\rV^2.
\end{equation}
Expanding \eqref{e19} and rearranging, we get
\begin{equation}\label{e20}
2\langle x-R(x),C(x)\rangle=\lV x\rV^2-\lV R(x)\rV^2.
\end{equation}
Substracting $2\langle x-R(x),T(x)\rangle$ from both sides of \eqref{e20} and using \eqref{e18},we obtain
$$
4\langle x-T(x),C(x)-T(x)\rangle = 2\langle x-R(x),C(x)-T(x)\rangle=\lV x\rV^2-\lV R(x)\rV^2 -2\langle x-R(x),T(x)\rangle
$$
$$
=\lV x\rV^2-\lV R(x)\rV^2 -\langle x-R(x),x+R(x)\rangle =0,
$$
which implies the result.
\end{proof}

Next we establish a basic property of the circumcenter, which ensures that the CRM sequence, starting at a point in $U$, remains in $U$. 

\begin{proposition}\label{pp9} 
If $z\in U$ then $C(z)\in U$.
\end{proposition}

\begin{proof} 
We consider three cases.
If $R(z)\in U$ then $R_U(R(z))=R(z)$, in which case $z,R(z),R_U(R(z))\in U$, so that the affine hull of these
three points is contained in $U$. Since by definition $C(z)$ belongs to this affine hull, the result holds.

If $z=P_U(R(z))$ then the affine hull of $\{z,R(z),R_U(R(z))\}$ is the line determined by $z$ and $R(z)$ and $C(z)=$ 
circ$\{z,R(z),R_U(R(z))\}=P_U(R(z))=z\in U$, so that the result holds. 

Assume that $z\ne P_U(R(z))$ and that $R(z)\notin U$. We claim that
$C(z)$ belongs to the line passing through $z$ and $P_U(R(z))$. Observe that, since $\lV C(z)-R(z)\rV=\lV C(z)-R_U(R(z))\rV$,
$C(z)$ belongs to the hyperplane orthogonal to $R(z)-R_U(R(z))$ passing through $\frac{1}{2}(R(z),R_U(R(z)))=P_U(R(z))$, say $H$.
On the other hand, by definition, $C(z)$ belongs to the affine manifold $E$ spanned by $z,R(z),R_U(z)$. So, $C(z)\in E\cap U$.
Since $R(z)\notin U$, dim$(E\cap U)<$ dim$(E)\le 2$. Note that $P_U(z)=\frac{1}{2}\left(R(z)+R_U(R(z))\right)=P_U(R(z))$ belongs to $E$.
Hence the line through $z,P_U(R(z))$, say $L$, is contained in $E$, and by a dimensionality argument we conclude that 
$L=E$. Since $C(z)\in E$, we get that $C(z)\in L$. Since $z,P_U(R(z))$ belong to $U$, we have that $C(z)\in L\subset U$, completing the proof.   
\end{proof}  

We continue with an important intermediate result.

\begin{proposition}\label{p10}
Consider the operators $C,S:\re^n\to\re^n$ defined above. Then $S(x)$ belongs to the segment between 
$x$ and $C(x)$ for all $x\in U$.
\end{proposition}
\begin{proof}
Let $E$ denote the affine manifold spanned by $x, R(x)$ and $R_U(R(x))$. By definition, the circumcenter of these
three points, namely $C(x)$, belongs to $E$. We claim that $S(x)$ also belongs to $E$. We proceed to prove the claim.
Since $U$ is an affine manifold, $P_U$ is an affine operator, so that $P_U(\alpha x+(1-\alpha)x')=\alpha P_U(x)+(1-\alpha)P_U(x')$
for all $\alpha\in\re$ and all $x,x'\in\re^n$. Thus $R_U(R(x))=2P_U(R(x))-R(x)$, so that 
\begin{equation}\label{e21}
P_U(R(x))=\frac{1}{2}\left(R_U(R(x))+R(x)\right).
\end{equation}
On the other hand, using the affinity of $P_U$, the definition of $S$ and the assumption that $x\in U$, we have
\begin{equation}\label{e22} 
P_U(R(x))=P_U(2T(x)-x)=2P_U(T(x))-P_U(x)=2S(x)-x,
\end{equation}
so that
\begin{equation}\label{e23}
S(x)=\frac{1}{2}\left(P_U(R(x))+x\right).
\end{equation}
Combining \eqref{e21} and \eqref{e23},
$$
S(x)=\frac{1}{2}x+\frac{1}{4}R_U(R(x))+\frac{1}{4}R(x),
$$ 
i.e., $S(x)$ is a convex combination of $x, R_U(R(x))$ and $R(x)$. Since these three points belong to $E$, the same
holds for $S(x)$ and the claim holds. 

We observe now that $x\in U$ by assumption, $S(x)\in U$ by definition,
and $C(x)\in U$ by Proposition \ref{pp9}. Now we consider three cases: if dim$(E\cap U)=0$ then $x,S(x)$ and $C(x)$
coincide and the result holds trivially. If dim$(E\cap U)=2$ then $E\subset U$, so that $R(x)\in U$ and hence
$R_U(R(x))=R(x)$, in which case $C(x)$ is the midpoint between $x$ and $R(x)$, which is precisely $T(x)$.
Hence, $T(x)\in U$, so that $S(x)=P_U(T(x))=B(x)=C(x)$, implying that $S(x)$ and $C(x)$ coincide, and the
result holds trivially. The interesting case is the remaining one, i.e., dim$(E\cap U)=1$. In this case $x, S(x)$ and
$C(x)$ lie in a line, so that we can write $C(x)=x+\eta(S(x)-x)$ with $\eta\in\re$, and it suffices to prove that $\eta\ge 1$.

By the definition of $\eta$, 
\begin{equation}\label{e24}
\lV C(x)-x\rV=\lv\eta\rv\,\lV T(x)-x\rV.
\end{equation}
Since $C(x)\in U$, nonexpansiveness of $P_U$ implies that
\begin{equation}\label{e25}
\lV C(x)-R(x)\rV\ge\lV C(x)-P_U(R(x))\rV.
\end{equation}
Then
$$
\lV C(x)-x\rV=\lV C(x)-R(x)\rV\ge\lV C(x)-P_U(R(x))\rV
=\lV\left(C(x)-x\right)-\left(P_U(R(x))-x\right)\rV
$$
\begin{equation}\label{e26}
=\lV\eta\left(S(x)-x\right)-2\left(S(x)-x\right)\rV=\lv\eta-2\rv\,\lV S(x)-x\rV,
\end{equation}
using the definition of the circumcenter in the first equality, \eqref{e25} in the inequality, and 
the definition of $\eta$ and $S$ in the third equality. Combining \eqref{e24} and \eqref{e26},
we get 
$$
\lv\eta\rv\,\lV S(x)-x\rV\ge\lv\eta-2\rv\,\lV S(x)-x\rV,
$$
implying that $\lv\eta\rv\ge\lv 2-\eta\rv$,
which holds only when $\eta\ge 1$, completing the proof.
\end{proof}

We continue with a key result for the convergence analysis of CRM, comparing the behavior of the CRM and the MAP operators. 
Again the argument in this proof differs from the case of CRM applied to MAP, 
presented in \cite{ABBIS}.

\begin{proposition}\label{p11}
With the notation of Proposition \ref{p10}, 
for all $y\in {\rm Fix}(T,P_U)$ and all $z\in U$, it holds that
\begin{itemize}
\item[i)]  $\lV C(z)-y\rV\le\lV S(z)-y\rV$,
\item[ii)] 
${\rm dist}\left(C(z),{\rm Fix}(T,P_U)\right)\le {\rm dist}\left(S(z),{\rm Fix}(T,P_U)\right)$,
\end{itemize}
\end{proposition}

\begin{proof}
\begin{itemize}
\item[i)]
Take $z\in U,y\in$ Fix$(T,P_U)$. If $z \in F(T)$, then the result follows
trivially, because then $P_U(T(z))=z=C(z)$ and there is nothing to prove. So, assume that $z\in U\setminus F(T)$. We claim that 
\begin{equation}\label{e27}
\lV P_U(T(z))-z\rV\le\lV T(z)-z\rV\le\lV C(z)-z\rV.
\end{equation}
For proving the first inequality in \eqref{e27},
we conclude, from the fact that $z\in U$ and an elementary property of orthogonal projections, that
\begin{equation}\label{e28}
\lV P_U(T(z))-z\rV\le\lV T(z)-z\rV.
\end{equation}
Since $R(z)=2T(z)-z$, we get that
\begin{equation}\label{e29}
\lV R(z)-z\rV=2\lV T(z)-z\rV.
\end{equation}
Using \eqref{e28} and \eqref{e29},
$$
\lV T(z)-z)\rV= \frac{1}{2}\lV R(z)-z\rV=
\frac{1}{2}\lV(R(z)-C(z)+C(z)-z)\rV
$$
\begin{equation}\label{e30}
\le\frac{1}{2}\left(\lV R(z)-C(z)\rV+\lV C(z)-z\rV\right)
=\frac{1}{2}\left(\lV z-C(z)\rV+\lV C(z)-z\rV\right)
=\lV C(z)-z\rV.
\end{equation}
where the third equality holds because $C(z)$ is equidistant from $z,R(z),$ and $R_U(R(z)).$ 
The claim follows then from \eqref{e27} and \eqref{e30}.
 
By Proposition \ref{p10}, $T(z)$ belongs to the segment between $z$ and $C(z)$, i.e., there exists 
$\alpha\in [0,1]$ such that
$S(z)=\alpha C(z)+(1-\alpha)z$ and $\alpha<1$ because $z\notin F(T)$, so that
\begin{equation}\label{e31}
S(z)-C(z)=\dfrac{1-\alpha}{\alpha}(z-S(z)).
 \end{equation}
Note that
\begin{equation}\label{e32}
\langle z-S(z),C(z)-y\rangle =\langle z-T(z),C(z)-T(z)\rangle+\langle z-T(z),T(z)-y\rangle+
 \langle T(z)-S(z),C(z)-y\rangle.
\end{equation}		
Now we look at the three terms in the right hand side of \eqref{e32}. The first one vanishes as a consequence of Proposition \ref{p9}. 
The third one vanishes because $S(z)=P_U(T(z))$, and $U$ is an affine manifold, so that $T(z)-S(z)$ is orthogonal to any vector in $U$, as is the case for
$C(z)-y$, since $y\in U$ by assumption and $C(z)\in U$ by Proposition \ref{pp9}. The second term is nonnegative by Proposition \ref{pp1}.     
It follows hence from \eqref{e32} that
\begin{equation}\label{e35}
\langle z-S(z),C(z)-y\rangle\ge 0.
\end{equation}

 Now, \eqref{e35} together with \eqref{e31} gives us
 \begin{equation}\label{e36}
 \langle S(z)-C(z),y-C(z)\rangle=\frac{1-\alpha}{\alpha}\langle z-S(z),y-C(z)\rangle\le 0.
 \end{equation}
 It follows from \eqref{e36} that  $\lV C(z)-y\rV\le\lV S(z)-y\rV$ for all $y\in$ Fix$(T,P_U)$ and all $z\in U$, 
establishing (i).
 
\item[ii)]
Let $\bar{z}, \hat{z}\in$ Fix$(T,P_U)$ realize the distance from $C(z),S(z)$
to Fix$(T,P_U)$ respectively. Then, in view of (i),
$$
{\rm dist}(C(z),{\rm Fix}(T,P_U))=\lV C(z)-\bar z\rV\le\lV C(z)-\hat z\rV\le\lV S(z)-\hat z\rV={\rm dist}(S(z),{\rm Fix}(T,P_U))
$$
proving (ii).
\end{itemize}
\end{proof}
  
Next we complete the convergence analysis of CRM applied to FPP. Here again, the proofline differs from the one in \cite{ABBIS},
where a specific property of orthogonal projections was used to characterize $C(z)$ as the projection onto a certain set, which
does not work when $T$ is an arbitrary firmly nonexpansive operator. 	
	
\begin{Thm}\label{t1}
Let $T:\re^n\to\re^n$ be a firmly nonexpansive operator and $U\subset\re^n$ an affine manifold. Assume that {\rm Fix}$(T,P_U)\ne\emptyset$.
Let $\{x^k\}$ be the sequence generated by CRM for solving FPP$(T,P_U)$, i.e., $x^{k+1}=C(x^k)$. If $x^0\in U$, then $\{x^k\}$ 
is contained in $U$ and converges to a point in {\rm Fix}$(T,P_U)$.
\end{Thm}
 
\begin{proof}
The fact that $\{x^k\}\subset U$ results from invoking Proposition \ref{pp9} in an inductive way, starting
with the assumption that $x^0\in U$. 

Take any $y\in$ Fix$(T,P_U)$ Then,
\begin{equation}\label{e37}
\lV x^{k+1}-y\rV^2=\lV C(x^k)-y\rV^2
\le\lV S(x^k)-y\rV^2\le\lV x^k-y\rV^2-\lV S(x^k)-x^k\rV^2 
\end{equation}
where the first inequality follows from Proposition \ref{p11}(i), and the second one follows from Proposition \ref{p7}, since $P_U(x^k)=x^k$
by Proposition \ref{pp9} and $S=P_U\circ T$.

\eqref{e37} says that $\{x^k\}$ is Fej\'er monotone with respect to Fix$(T,P_U)$, and the remainder of the proof is standard.
By \eqref{e37},
 $\{x^k\}$ is bounded and $\{\lV x^k-y\rV\}$ is nonincreasing and nonnegative, hence
convergent, for all $y\in$ Fix$(T,P_U)$. It follows also from \eqref{e37} that 
\begin{equation}\label{e38}
\lim_{k\to\infty} S(x^k)-x^k=0.
\end{equation}
Let $\bar x$ be any cluster point of  $\{x^k\}$. Taking limits in \eqref{e38} along a subsequence converging to $\bar x$,
we conclude that $S(\bar x)=\bar x$, i.e., $\bar x\in F(S)=$ Fix$(T,P_U)$, so that all cluster points of $\{x^k\}$ belong to Fix$(T,P_U)$.
Looking now \eqref{e37} with $\bar x$ substituting for $y$, we get that $\{\lV x^k-\bar x\rV\}$ is a nonincreasing sequence with a subsequence
converging to $0$, so that the whole sequence $\{\lV x^k-\bar x\rV\}$ converges to $0$. It follows that $\bar x$ is the unique cluster point
of $\{x^k\}$, so that $\lim_{k\to\infty}x^k=\bar x\in$ Fix$(T,P_U)$.
\end{proof}

For future reference, we state the Fej\'er monotonicity of $\{x^k\}$  with respect to Fix$(T,P_U)$ as a corollary.
\begin{Cor}\label{c3}
With the notation of Theorem \ref{t1}, 
$\lV x^{k+1}-y\rV^2
\le\lV x^k-y\rV^2-\lV S(x^k)-x^k\rV^2$ for all $y\in {\rm Fix}(T,P_U)$ and all $k\in\na$.
\end{Cor}
\begin{proof} 
The result follows from \eqref{e37}.
\end{proof}

\section{Linear convergence of CRM applied to FPP under an error bound condition}\label{s4}

In \cite{ABBIS}, when dealing with CFP with two convex sets, namely $K,U$, the following  {\it global error bound}, 
which we will call {\bf EB1}, was considered:

\medskip

\noindent {\bf EB1}: There exists $\bar\omega >0$ such that ${\rm dist}(x,K)\ge\bar\omega\, {\rm dist}(K\cap U)$ for all $x\in U$.

\medskip

Let us comment on the connection between
{\bf EB1} and other notions of error bounds which have been introduced in the past, all of them related to regularity assumptions 
imposed on the solutions of certain problems. If the problem at hand consists of solving $H(x)=0$ with a smooth 
$H:\re^n\to\re^m$, a classical regularity condition demands that $m=n$ and the Jacobian matrix of $H$ be 
nonsingular at a solution $x^*$, in which case, Newton's method, for instance, is known to enjoy superlinear or quadratic 
convergence. This condition implies local uniqueness of the solution $x^*$. For problems
with nonisolated solutions, a less demanding assumption is the
notion of {\it calmness} (see \cite{RoW}, Chapter 8, Section F), which requires that 
\begin{equation}\label{e39}
\frac{\lV H(x)\rV}{{\rm dist}(x,S^*)}\ge\omega
\end{equation}
for all $x\in\re^n\setminus S^*$ and some $\omega>0$, where $S^*$ is the solution set, i.e., the set of zeros of $H$.
Calmness, also called upper-Lipschitz continuity (see \cite{Rob}), is a classical example of error bound, and it holds 
in many situations, e.g., when $H$ is affine, by virtue of Hoffman's Lemma, (see \cite{Hof}). It implies that the solution set 
is locally a Riemannian manifold (see \cite{BeI}), and it has been used for establishing superlinear convergence of 
Levenberg-Marquardt methods in \cite{KYF}.

When dealing with convex feasibility problems, it seems reasonable to replace the numerator of \eqref{e39} 
by the distance from $x$ to some of the convex sets, as was done for instance, in \cite{ABBIS}, giving rise to {\bf EB1}. 
In \cite {ABBIS}, it was proved that under {\bf EB1}, MAP converges linearly, with asymptotic constant bounded above
by $\sqrt{1-\bar\omega^2}$, and that CRM also converges linearly, with a better upper bound for the asymptotic constant,
namely $\sqrt{(1-\bar\omega^2)/(1+\bar\omega^2)}$. In this section we will prove that in the FPP case both sequences 
converge linearly, with asymptotic constant bounded by $\sqrt{1-\bar\omega^2}$.

In the case of FPP, dealing with a firmly nonexpansive $T:\re^n\to\re^n$, and an affine manifold $U\subset\re^n$, the appropriate 
error bound turns out to be:

\medskip

\noindent {\bf EB}: There exists $\omega >0$ such that $\lV x-T(x)\rV\ge\omega\, {\rm dist}(x,{\rm Fix}(T,P_U)$ for all $x\in U$.

\medskip

We mention here that it suffices to consider an error bound less demanding than {\bf EB}, namely a local one, where the inequality
above is requested to hold only for points in $U\cap V$, where $V$ is a given set, e.g., a ball around the limit of 
the sequence generated by the algorithm, 
assumed to be convergent. An error bound of this type was used in \cite{AABBIS}. We refrain to do so just for the sake of a simpler exposition.

\begin{proposition}\label{p12}
Let $T:\re^n\to\re^n$ be a firmly nonexpansive operator, $U\subset\re^n$ an affine manifold and $C,S:\re^n\to\re^n$ the CRM and the MAP operators
respectively. 
Assume that {\rm Fix}$(T,P_U)\ne\emptyset$ and that 
{\bf EB} holds. Then
 \begin{equation}\label{e40}
{\rm dist}(C(x),{\rm Fix}(T,P_U))^2\le {\rm dist}(S(x),{\rm Fix}(T,P_U))^2\le (1-\omega^2)dist(x,{\rm Fix}(T,P_U))^2,
 \end{equation}  
 for all $x\in U$, with $\omega$ as in {\bf EB}.
\end{proposition}
\begin{proof}
 First note that if $x\in F(T)$, then \eqref{e40} holds trivially, so that we assume from now on that $T(x)\ne x$. Take any $y\in$ Fix$(T,P_U)$. 
Since $T$ is firmly nonexpansive and $y\in F(T)$, we have  
\begin{equation}\label{e41}
\lV x-y\rV^2\ge\lV T(x)-T(y)\rV^2+\lV (x-y)-(T(x)-T(y))\rV^2=\lV T(x)-y\rV^2+\lV x-T(x)\rV^2,
\end{equation}
We take now a specific point in Fix$(T,P_U)$, namely $\bar y=P_{{\rm Fix}(T,P_U)}(x)$, and rewrite {\bf EB} as
\begin{equation}\label{e42}
\lV x-T(x)\rV^2\ge\omega^2\lV x-\bar y\rV^2.
\end{equation}
Combining \eqref{e41} and \eqref{e42},we get
\begin{equation}\label{e43}
\lV x-\bar y\rV^2\ge\lV x- T(x)\rV^2+\lV T(x)-\bar y\rV^2 
\ge\omega^2\lV x-\bar y\rV^2+\lV T(x)-\bar y\rV^2.
\end{equation}
Rearranging \eqref{e43}, we conclude that
\begin{equation}\label{e44} 
(1-\omega^2) \lV x-\bar y\rV^2\ge\lV T(x)-\bar y\rV^2.
\end{equation}

Note that $\langle T(x)-S(x),\bar y-T(x)\rangle=\langle T(x)-P_U(T(x)),\bar y-T(x)\rangle\le 0$, by an elementary property of 
orthogonal projections, since $\bar y\in U$.
Hence,
\begin{equation}\label{e45}
\lV T(x)-\bar y\rV^2\ge\lV T(x)-S(x) \rV^2+\lV S(x)-\bar y\rV^2.
\end{equation}
Let $\hat y=P_{{\rm Fix}(T,P_U)}(S(x))$.
From \eqref{e44} and \eqref{e45} we obtain
\begin{equation}\label{e46}
(1-\omega^2)\lV x-\bar y\rV^2\ge\lV T(x)-\bar y\rV^2\ge\lV T(x)-S(x)\rV^2+\lV S(x)-\bar y\rV^2
\ge\lV S(x)-\bar y\rV^2\ge\lV S(x)-\hat y\rV^2,
\end{equation}
where the second inequality holds by \eqref{e45} and the last one follows from the definition of orthogonal projection. 
From \eqref{e46} we conclude, recalling the definitions of $\bar y,\hat y$, that
\begin{equation}\label{e47}
{\rm dist}(S(x),{\rm Fix}(T,P_U))^2\le(1-\omega^2){\rm dist}(x,{\rm Fix}(T,P_U))^2,
\end{equation}
which shows that the second inequality in \eqref{e40} holds. Next we look at the first one.  
Let $\tilde y=P_{{\rm Fix}(T,P_U)}(C(x))$. We have that
\begin{equation}\label{e48}
\lV C(x)-\tilde y\rV^2\le\lV C(x)-\hat y\rV^2\le\lV S(x)-\hat y\rV^2
\le\lV S(x)-\bar y\rV^2
\le(1-\omega^2)\lV x-\bar y\rV^2,
\end{equation}
where the first and the third inequality hold by the definition of orthogonal projection, 
the second one from Proposition \ref{p11}(i) and the last one holds by \eqref{e46}. 
Note that the first inequality in \eqref{e40} follows immediately from \eqref{e48}, in view of the definitions of $\tilde y, \bar y$.
\end{proof}

 \begin{Cor}\label{c4}
 Under the assumptions of Proposition \ref{p12}, let $\{z^k\},\{x^k\}$ be the sequences generated by MAP and CRM respectively,
for solving FPP$(T,P_U)$, i.e.,  $z^{k+1}=S(z^k),$ and $x^{k+1}=C(x^k),$ starting from some $z^0\in \re^n$ and $x^0\in U$. 
Then the scalar sequences $\{a^k\}, \{b^k\}$, defined as $a^k={\rm dist}(z^k,{\rm Fix}(T,P_U))$ and $b^k={\rm dist}(x^k,{\rm Fix}(T,P_U))$, 
converge Q-linearly to zero with asymptotic constants bounded above by $\sqrt{1-\omega^2},$ with $\omega$ as in {\bf EB}. 
 \end{Cor}
\begin{proof}
It follows from \eqref{e40} that, for all $x\in U$,
\begin{equation}\label{e49}
{\rm dist}(S(x),{\rm Fix}(T,P_U))^2\le(1-\omega^2){\rm dist}(x,{\rm Fix}(T,P_U))^2,
\end{equation}
and  that, for all $z\in U$,
\begin{equation}\label{e50}
{\rm dist}(C(x),{\rm Fix}(T,P_U))^2\le(1-\omega^2){\rm dist}(x,{\rm Fix}(T,P_U))^2,
\end{equation}
In view of the definitions of $\{x^k\}, \{z^k\}$, and remembering that both sequences are contained in $U$, by Proposition \ref{pp9}
in the case of $\{x^k\}$ and by definition of $S$ in the case of $\{z^k\}$, 
we get from \eqref{e49}, \eqref{e50},
\begin{equation}\label{e51}
\frac{{\rm dist}(z^{k+1},{\rm Fix}(T,P_U))}{{\rm dist}(z^k,{\rm Fix}(T,P_U))}\le\sqrt{1-\omega^2},
\end{equation}
\begin{equation}\label{e52}
\frac{{\rm dist}(x^{k+1},{\rm Fix}(T,P_U))}{{\rm dist}(x^k,{\rm Fix}(T,P_U))}\le\sqrt{1-\omega^2}.
\end{equation}
The result follows immediately from \eqref{e51}, \eqref{e52}.
\end{proof}

Note that the results of Corollary \ref{c4} do not entail immediately that the sequences $\{x^k\}, \{z^k\}$ themselves
converge linearly; a sequence $\{y^k\}$ may converge to a point $y\in M\subset\re^n$,
in such a way that $\{{\rm dist}(y^k,M)\}$ converges linearly to $0$ but $\{y^k\}$ itself converges sublinearly.
Take for instance $M=\{(s,0)\in\re^2\}$, $y^k=\left(1/k,2^{-k}\right)$. This sequence converges
to $0\in M$, ${\rm dist}(y^k,M)=2^{-k}$ converges linearly to $0$ with asymptotic constant equal to $1/2$, but
the first component of $y^k$ converges to $0$ sublinearly, and hence the same holds for the sequence
$\{y^k\}$. The next  well known lemma establishes that this situation
cannot occur when $\{y^k\}$ is Fej\'er monotone with respect to $M$, i.e., $\lV y^{k+1}-y\rV\le\lV y^k-y\rV$ for all $y\in M$.

\begin{Lem}\label{l1} 
Consider $M\subset \re^n$, $\{y^k\}\subset\re^n$. Assume that $\{y^k\}$ is Fej\'er
monotone with respect to $M$, and that ${\rm dist}(y^k,M)$ converges R-linearly to $0$. Then $\{y^k\}$ converges 
R-linearly to some point $y^*\in M$, with asymptotic constant bounded above by the 
asymptotic constant of $\{{\rm dist}(y^k,M)\}$.
\end{Lem}

\begin{proof}
See, e.g., Lemma 1 in \cite{ABBIS}.
\end{proof}

We show next that the sequences $\{x^k\}$ and $\{z^k\}$ are  R-linearly convergent 
under Assumption {\bf EB}, with asymptotic constants bounded by $\sqrt{1-\omega^2}$, where $\omega$ is the {\bf EB} parameter.

\begin{Thm}\label{t2}
Let $T:\re^n\to\re^n$ be a firmly nonexpansive operator and $U\subset\re^n$ is an affine manifold. Assume that Fix$(T,P_U)\ne\emptyset$ and that 
condition {\bf EB} Holds.  Consider the sequences $\{z^k\},\{x^k\}$ generated by MAP and CRM respectively, for solving ${\rm Fix}(T,P_U)$, 
i.e., $x^{k+1}=S(x^k)$ and $z^{k+1}=C(z^k)$, 
starting from some $z^0\in \re^n$ and some $x^0\in U$. Then both sequences converge R-linearly to points in ${\rm Fix}(T,P_U)$, with asymptotic constants 
bounded above by $\sqrt{1-\omega^2},$ with $\omega$ as in assumption {\bf EB}.
\end{Thm}
\begin{proof}
By Corollary \ref{c4}, both scalar sequences $a^k={\rm dist}(z^k,{\rm Fix}(T,P_U))$ and $b^k={\rm dist}(x^k,{\rm Fix}(T,P_U))$ are Q-linearly convergent to $0$ with asymptotic constant bounded above by  $\sqrt{1-\omega^2}<1$, and hence
R-linearly convergent to 0, with the same asymptotic constant. By Corollary \ref{c3}, the sequence $\{x^k\}$ is Fej\'er 
monotone with respect to Fix$(T,P_U)$, and the same holds for the sequence $\{z^k\}$, in view of \eqref{et17}. By
Theorem \ref{t1}, both sequences converge to points in Fix$(T,P_U)$. Finally, by Lemma \ref{l1}, both sequences converge 
R-linearly convergent to their limit points in the intersection, with asymptotic constants bounded by $\sqrt{1-\omega^2}$.
\end{proof}

We mention that in \cite{ABBIS} we showed that for CFP under EB, CRM achieves an asymptotic constant of linear convergence better than MAP. 
We have not been able to prove such superiority in the case of FPP. However, the numerical results exhibited in Section \ref{s5} strongly suggest that
the asymptotic constant of CRM is indeed better than the MAP one. The task of establishing such theoretical superiority is left as an open problem.

\section{Numerical experiments}\label{s5} 

We report here numerical comparisons between CRM and PPM for solving FPP with $p$ firmly nonexpansive operators.

All operators in this section belong to the family studied in Section \ref{s2}, i.e., they are convex combinations of orthogonal
projections onto a finite number of closed and convex sets with nonempty intersection. In view of Proposition \ref{p3}(ii),
these operators are ensured to have fixed points. Hence, in view of Proposition \ref{p5} they are not orthogonal projections themselves. 

The construction of the problems is as follows: for each instance we choose randomly a number $r\in\{3,4,5\}$ ($r$ is the number of convex sets in the 
convex combination). Then we sample values $\lambda_1, \dots, \lambda_r\in(0,1)$ with uniform distribution. We define $\mu_i=\lambda_i/(\sum_{\ell=1}^r)$, 
and we take the firmly nonexpansive operator $T$ as $T=\sum_{i=1}^r\mu_iP_{\mathcal{E}_i}$, where $\mathcal{E}_i$ is an ellipsoid and $P_{\mathcal{E}_i}$
is the orthogonal projection onto it. 

The ellipsoid  $\mathcal{E}_i$ is of the form
$\mathcal{E}_i:=\{x\in\re^n:g_i(x)\le 0\}$, where $g_i:\re^n\to\re$ is given as $g_i(x)= x^t A_ix +2 (b^i)^tx-\alpha_i$, 
with $A_i\in\re^{n\times n}$ symmetric positive definite, $b^i\in\re^n$ and $0<\alpha_i\in\re$. 
 Each matrix $A_i$ is of the form $A_i = \gamma I+B_i^\top B_i$, with  $B_i \in\re^{n\times n}$, $\gamma \in \re_{++}$, where $I$ stands for the identity matrix.  The matrix $B_i$ is a sparse matrix sampled from the standard normal distribution  with sparsity density  $p=2 n^{-1}$  and each vector $b^i$ is sampled from the uniform distribution between $[0,1]$. We then choose each $\alpha_i$ so that $\alpha_i > (b^i)^\top Ab^i$, which ensures that $0$ belongs to every $\mathcal{E}_i$, so that the intersection of the ellipsoids is nonempty. As explained above, this ensures that each instance of FPP has solutions. 

In order to compute the projection onto the ellipsoids we use a version of the Alternating Direction Method
of Multipliers (ADMM) suited for this purpose; see \cite{JCH}. The stopping criterion for ADMM
is as follows: we stop the ADMM iterative process when
the norm of the difference between 2 consecutive ADMM iterates 
is less than $10^{-8}$. We also fix a maximum number of $10\, 000$ ADMM iterations. 

For CRM, we use Pierra's product space reformulation, as explained in Section \ref{s1}. We implement PPM directly from its definition (see Section \ref{s1}).
The stopping criterion for both CRM and PPM is similar to the one for the ADMM subroutine, but with a different tolerance:
the iterative process stops when
the norm of the difference between 2 consecutive CRM or PPM iterates 
is less than $10^{-6}$. The maximum number of iterations is fixed at $50\, 000$ for both algorithms. 

The experiments consists of solving, with CRM and PPM, $250$ instances of FPP selected 
as follows. We consider the following values for the dimension $n$:
$\{10,30, 50, 100, 200\}$, and for each $n$ we take $p$ firmly nonexpansive operators with $p\in\{10,25,50,100,200\}$. For each of these 25 pairs $(n,p)$,
we randomly generate 10 instances of FPP with the above explained procedure.

The initial point $x^0$ is of the form $(\eta,\dots,\eta)\in\re^n$, with $\eta<0$  and $\lv\eta\rv$ sufficiently large so as to guarante  
that $x^0$ is far from all the ellipsoids.

The computational experiments were carried out
on an Intel Xeon W-2133 3.60GHz with 32GB of RAM running Ubuntu 20.04. We implemented all 
experiments in Julia programming language v1.6 (see \cite{BEKS}). The codes of our experiments are fully available 
at: \url{
https://github.com/Mirza-Reza/FPP}

We report in Table \ref{table:fneellipsoidd_all} the following descriptive statistics for CRM and PPM: mean, maximum (max), minimum (min) 
and standard deviation (std) for iteration count (it) and CPU time in seconds (CPU (s)). In particular, the ratio of the CPU time 
(in average for all instances) of PPM with respect to CRM 
is $7.69$, meaning that CRM is, on the average, almost eight times faster that PPM.

\begin{table}[ht]
  \centering
  \caption{Statistics for all instances, reporting number of iterations and CPU time}
  \label{table:fneellipsoidd_all}
  \sisetup{
table-parse-only,
table-figures-decimal = 4,
table-format = +3.4e-1,
table-auto-round,
% output-exponent-marker = \text{e},
% 
}
\bigskip
\begin{tabular}{lrSccS}

\toprule  \textbf{Method} & &   {\textbf{mean}} &  {\textbf{max}} &   {\textbf{min}} &   {\textbf{std}} \\
\cmidrule(lr){3-6}
% \cmidrule(lr){1-1}
% \midrule
CRM & \texttt{it} & 144.288   & 554           &23          & 95.2581   \\
 & \texttt{CPU(s)} & \num{14.6048} & \num{120.3020}  & \num{0.2729}   &\num{ 22.4890}  \\
\cmidrule(lr){2-6}
PPM &  \texttt{it} & 5977.352   & 25000   &209       & 6385.9388   \\
 & \texttt{CPU(s)} & \num{ 112.3315} &  1085.9685& \num{1.2483}   & \num{190.3078}   \\

\bottomrule
\end{tabular}
\end{table}

We report next similar statistics, but separately for each dimension $n$.
Looking at Table \ref{table:fneellipsoid_n}, we observe that the CPU time for PPM grows linearly with the dimension $n$, while 
the growth of the CRM CPU time is somewhat higher than linear. As a consequence, the superiority of CRM over PPM, measured in terms 
of the quotient between the PPM CPU time and the CRP CPU time, is slightly decreasing with $n$: it goes from a ratio of  $9.17$
for $n=10$ to a ratio of $7.56$ for $n=200$. This said, it is clear that CRM vastly outperforms PPM in terms of CPU time for 
all the values of $n$ tested in our experiments.   

\newpage

\begin{table}[ht]
  \centering
  \caption{Statistics for instances of each dimension $n$, reporting number of iterations and CPU time}
  \label{table:fneellipsoid_n}
  \sisetup{
table-parse-only,
table-figures-decimal = 2,
table-format = +3.4e-1,
table-auto-round,
% output-exponent-marker = \text{e},
% 
}
\bigskip
\begin{tabular}{lcSccS}

\toprule
\textbf{Method} & &   {\textbf{mean}} &  {\textbf{max}} &   {\textbf{min}} &   {\textbf{std}} \\
\cmidrule(lr){3-6}
% \cmidrule(lr){1-1}
% \midrule
CRM & \texttt{it} & 141.84   & 512          & 28          & 99.10284758774593    \\
$n=10$& \texttt{CPU(s)} & \num{2.3247} & \num{6.8150}  & \num{0.2729}   &\num{  1.9756}  \\
\cmidrule(lr){2-6}
PPM &  \texttt{it} & 6024.54  & 19163    &209     & 6425.574393655403  \\
 $n=10$& \texttt{CPU(s)} & \num{21.3369} &  92.19569& \num{1.2483}   & \num{22.4132}   \\
\bottomrule
CRM & \texttt{it} & 153.5  & 526           & 46           & 92.21502046846815   \\
 $n=30$& \texttt{CPU(s)} & \num{ 4.6989} & \num{16.6523}  & \num{0.7607}   &\num{ 4.1754}  \\
\cmidrule(lr){2-6}
PPM &  \texttt{it} & 5608.44  & 18353    &500       &  5956.758800421585   \\
 $n=30$& \texttt{CPU(s)} & \num{42.9296} &  174.9737& \num{2.9861}   & \num{46.2969}   \\
\bottomrule
\\
CRM & \texttt{it} &129.5   & 469          & 23         & 91.70523431080693    \\
 $n=50$& \texttt{CPU(s)} & \num{6.8152} & \num{17.1668}  & \num{1.0480}   &\num{  5.0391}  \\
\cmidrule(lr){2-6}
PPM &  \texttt{it} & 5288.52  & 24680     &423       & 5548.505204971876  \\
 $n=50$& \texttt{CPU(s)} & \num{53.3709} &  222.7307& \num{3.5744}   & \num{55.7054}   \\
\bottomrule
\\
CRM & \texttt{it} & 152.04  & 399          & 28          & 84.19238920472563   \\
 $n=100$& \texttt{CPU(s)} & \num{15.5937} & \num{41.2581}  & \num{1.9661}   &\num{  12.4246}  \\
\cmidrule(lr){2-6}
PPM &  \texttt{it} & 7224.42   & 21978   &540       & 7663.860453035403  \\
 $n=100$& \texttt{CPU(s)} & \num{114.4037} &  428.8247& \num{6.3108}   & \num{108.4765}   \\
\bottomrule
\\
CRM & \texttt{it} & 144.56  & 554           & 42           & 105.72438886084895    \\
 n=200& \texttt{CPU(s)} & \num{43.5915} & \num{ 120.3019}  & \num{ 5.0157}   &\num{ 34.3053}  \\
\cmidrule(lr){2-6}
PPM &  \texttt{it} & 5740.84   & 22378      &370     & 5948.570740472034   \\
 n=200& \texttt{CPU(s)} & \num{  329.6167} &  1085.9685& \num{19.0842}   & \num{315.8783}   \\
\bottomrule
\end{tabular}
\end{table}

Next, we report in the next table similar statistics, but separately for problems involving $p$ firmly nonexpansive operators,
for each value of $p$.
Table \ref{table:fneellipsoid-Random p} indicates that both the CRM and the PPM CPU time grow slightly less that linearly in $p$, 
the number of firmly nonexpansive operators in each instance of FPP, but the 
growth in both cases seems to become linear for $p\ge 50$. Consistently with this behavior, the ratio between the PPM CPU time and the the CRM CPU 
time is about $3$ for $p=10,25$ and about $8$ for $p=50,100,200$. Again, for all values of $p$, CRM turns out to be highly better than PPM 
in terms of CPU time.

\newpage

\begin{table}[ht]
  \centering
  \caption{Statistics for instances of FPP problems with $p$ firmly nonexpansive operators, reporting number of iterations and CPU time}
  \label{table:fneellipsoid-Random p}
  \sisetup{
table-parse-only,
table-figures-decimal = 2,
table-format = +3.4e-1,
table-auto-round,
% output-exponent-marker = \text{e},
% 
}
\bigskip
\begin{tabular}{lcSccS}

\toprule
\textbf{Method} & &   {\textbf{mean}} &  {\textbf{max}} &   {\textbf{min}} &   {\textbf{std}} \\
\cmidrule(lr){3-6}
% \cmidrule(lr){1-1}
% \midrule
fneCRM & \texttt{it} & 91.0  & 263         & 28         & 50.174495513158874   \\
 $p=10$& \texttt{CPU(s)} & \num{ 2.8569} & \num{13.1619}  & \num{0.2729}   &\num{   2.8807}  \\
\cmidrule(lr){2-6}
PPM &  \texttt{it} & 1316.68  & 6765   &209    & 1264.5452849147  \\
 $p=10$& \texttt{CPU(s)} & \num{ 13.1578} &  50.8767& \num{1.2483}   & \num{11.4271}   \\
\bottomrule

%\textbf{Method} & &   {\textbf{mean}} &  {\textbf{max}} &   {\textbf{min}} &   {\textbf{std}} \\
%\cmidrule(lr){3-6}
% \cmidrule(lr){1-1}
% \midrule
CRM & \texttt{it} & 113.7  & 469       & 36          & 83.5955142337195  \\
 $p=25$& \texttt{CPU(s)} & \num{  6.5062} & \num{ 45.2021}  & \num{0.6664}   &\num{ 8.9416}  \\
\cmidrule(lr){2-6}
PPM &  \texttt{it} &2865.92 & 14617   &650     &   2651.0789036918536  \\
$p=25$& \texttt{CPU(s)} & \num{34.6541} &  242.2805& \num{ 2.9785}   & \num{ 47.5093}   \\
\bottomrule
\\
CRM & \texttt{it} &128.8  & 331        & 23        & 76.92437845052763   \\
 $p=50$& \texttt{CPU(s)} & \num{  10.43880} & \num{46.8045}  & \num{1.3100}   &\num{  11.8677}  \\
\cmidrule(lr){2-6}
PPM &  \texttt{it} & 4949.42 & 25000   & 870     & 5531.401251364793 \\
 $p=50$& \texttt{CPU(s)} & \num{ 88.4859} &  602.8599& \num{6.5347}   & \num{125.0821}   \\
\bottomrule
\\
CRM & \texttt{it} & 166.28  & 526         & 49         &  91.18882387661331  \\
 $p=100$& \texttt{CPU(s)} & \num{ 18.8065} & \num{70.6532}  & \num{ 2.4265}   &\num{  20.1719}  \\
\cmidrule(lr){2-6}
PPM &  \texttt{it} & 7077.46  & 25000   & 1586       & 4970.777481279966  \\
 $p=100$& \texttt{CPU(s)} & \num{143.0699} &   729.1966& \num{12.0125}   & \num{ 171.2874}   \\
\bottomrule
\\
CRM & \texttt{it} & 221.66 & 554         & 88         & 105.57890130134903    \\
 $p=200$& \texttt{CPU(s)} & \num{ 34.4157} & \num{ 120.3019}  & \num{ 4.7277}   &\num{  35.5202}  \\
\cmidrule(lr){2-6}
PPM &  \texttt{it} & \num{13677.28 } & 25000      &  4015    & 6856.3914  \\
 $p=200$& \texttt{CPU(s)} &    282.2900 &  1085.9685 & \num{ 31.8832}   & \num{   295.7094}   \\
\bottomrule
\end{tabular}
\end{table}

Finally, we exhibit the performance profile, in the sense of \cite{DoM}, for all the instances. Again, the superiority of CRM with respect
to PPM is fully corroborated. 

\begin{figure}[ht]
    \centering
    \includegraphics[scale=0.9]{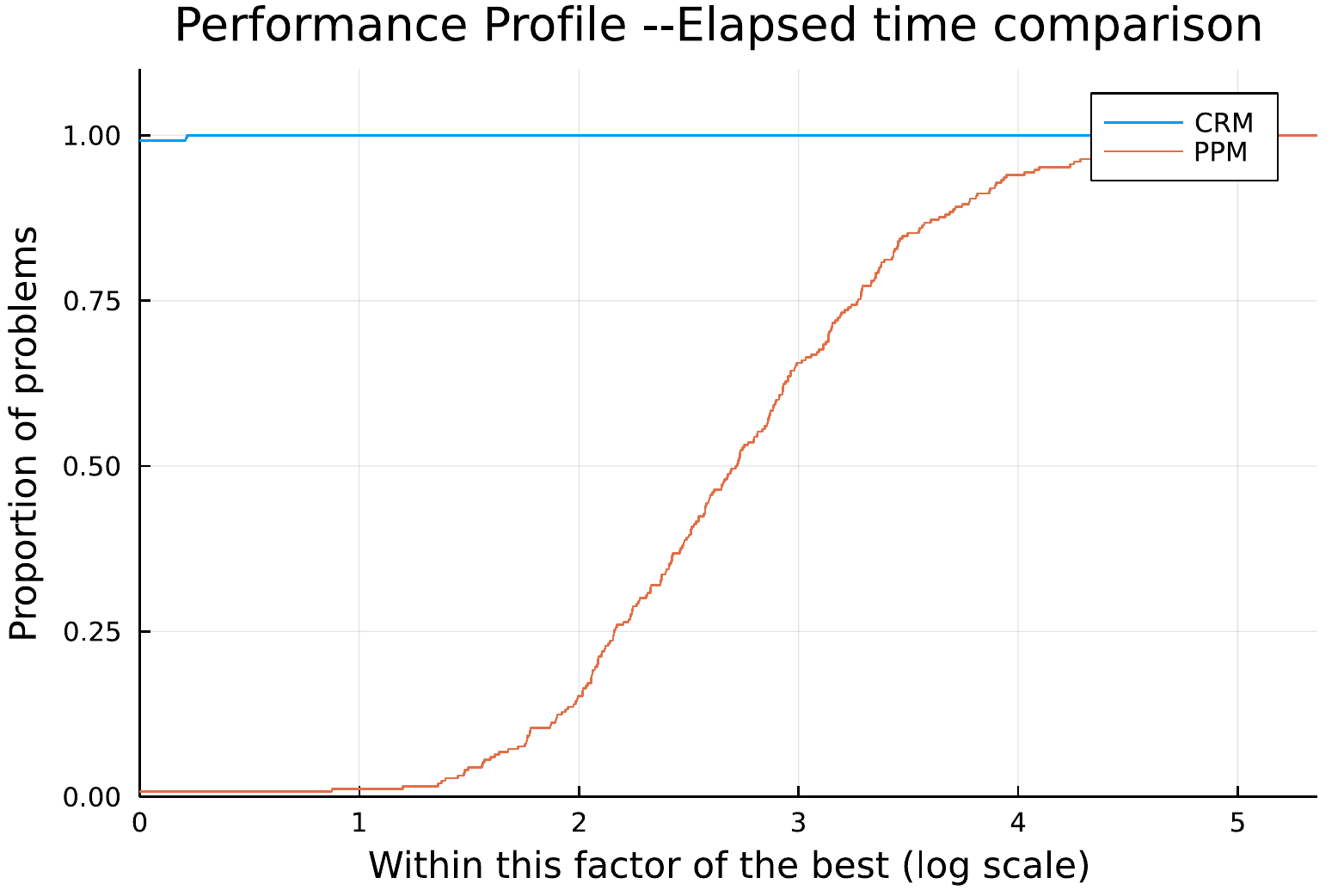}
    \caption{Performance proﬁle of experiments with ellipsoidal feasibility – CRM vs PPM}
    \label{fig:performance-profile-ppm}
\end{figure}


\newpage
\begin{thebibliography}{10}

\bibitem{ACT}
Arag\'on Artacho, F.J., Campoy, R., Tam, M.K. The
Douglas-Rachford algorithm for convex and nonconvex
feasibility problems.
{\it Mathematical Methods of Operations Research\/} {\bf 91} (2020) 201-240.

\bibitem{AABBIS} Ara\'ujo, G, Arefidamghani, R., Behling, R., Bello Cruz, Y., Iusem, A.N., Santos, L.-R.
Circumcentering approximate reflexions for solving the convex feasibility problem.
{\it Fixed Point Theory and Algorithms in Science and Engineering\/} (2022) Article no. 1.

\bibitem{ABBIS} Arefidamghani, R., Behling, R., Bello Cruz, J.Y., Iusem, A.N.,
Santos, L.R. The circumcentered-projection method achieves better rates than alternating projections.
{\it Computational Optimization and Applications\/} {\bf 79} (2021) 507-530.

\bibitem{BOW1}
Bauschke, H.H., Ouyang, H., Wang, X. On circumcenters of finite sets in
Hilbert spaces.
{\it Linear and Nonlinear Analysis\/} {\bf 4} (2018) 271--295.

\bibitem{BOW2}
Bauschke, H.H., Ouyang, H., Wang, X. Circumcentered methods induced by
isometries.
{\it Vietnam Journal of Mathematics\/} {\bf 48} (2020).

\bibitem{BOW3}
Bauschke, H.H., Ouyang, H., Wang, X. On circumcenter mappings induced by
nonexpansive operators. {\it Pure and Applied Functional Analysis\/}
{\bf 6} (2021) 257-288.

\bibitem{BOW4}
Bauschke, H.H., Ouyang, H., Wang, X. 
On the linear convergence of
circumcentered isometry methods. {\it Numerical Algorithms\/} {\bf 87} (2021) 268-297.

\bibitem{BOW5} Bauschke, H.H., Ouyang, H., Wang, X. Best Approximation mappings in Hilbert spaces.
To be published in {\it Mathematical Programming\/}. arXiv:2007.02644 (2020).

\bibitem{BBS1}
Behling, R., Bello Cruz, J.Y., Santos, L.R. On the linear convergence of the
circumcentered-reflection method.
{\it Operations Research Letters\/} {\bf 46} (2018) 159--162.

\bibitem{BBS2}
Behling, R., Bello Cruz, J.Y., Santos, L.R. Circumcentering the
Douglas-Rachford method.
{\it Numerical Algorithms\/} {\bf 78}  (2018) 759--776.

\bibitem{BBS3}
Behling, R., Bello Cruz, J.Y., Santos, L.R. The block-wise
circumcentered- reflection method.
{\it Computational Optimization and Applications\/} {\bf 76} (2020) 675--699.

\bibitem{BBS4}
Behling, R., Bello Cruz, J.Y., Santos, L.R. On the
Circumcentered-Reflection Method for the Convex Feasibility
Problem.
{\it Numerical Algorithms\/}  (2020).

\bibitem{BeI} Behling, R., Iusem, A.N. The effect of calmness on the solution
of systems of nonlinear equations. {\it Mathematical
Programming\/} {\bf 137} 155-165 (2013).

\bibitem{BEKS} Bezanson, J., Edelman, A., Karpinski, S., Shah, V.B. Julia: A fresh
approach to numerical computing. {\it SIAM Review\/} {\bf  59} (2017) 65–98.

\bibitem{CeS} Censor, Y., Segal, A. The split common fixed point problem for directed operators.
{\it Journal of Convex Analysis\/} {\bf 16} (2009) 587-600.

\bibitem{CeZ} Censor, Y., Zenios, S. {\it Parallel Optimization: Theory, Algorithms and Applications\/}.
Oxford University Press, Bew York (1998).
 
\bibitem{Cim} Cimmino, G. Calcolo approssimato per le soluzione dei sistemi
de equazioni lineari. {\it La Ricerca Scientifica\/} {\bf II-16} (1938)
326-333.

\bibitem{DeI}
De Pierro, A.R., Iusem, A.N. A parallel projections method
for finding a common point of a family of convex sets. {\it Pesquisa
Operacional\/} {\bf 5} (1985) 1-20.

\bibitem{DHL1}
Dizon, N., Hogan, J., Lindstrom, S.B. 
Circumcentering reflection methods
for nonconvex feasibility problems. arXiv 1910.04384 (2019). 

\bibitem{DHL2} Dizon, N., Hogan, J., Lindstrom, S.B. 
Centering projection methods for wavelet feasibility problem.
arXiv: 2005.05687 (2020)

\bibitem{DoM} 
Dolan, E.D., Mor \'e, J.J. Benchmarking optimization software with performance profiles. 
{\it Mathematical Programming\/} {\bf 91} (2002) 201–213.

\bibitem{Hof} Hoffman, A.J. On approximate solutions of systems of
linear equations. {\it Journal of Research of the National Bureau of
Standards\/} {\bf 49} (1952) 263-265.
 
\bibitem{IuD}
Iusem, A.N., De Pierro, A.R. On the set of weighted least squares
solutions of systems of linear inequalities. {\it Commentationes
Mathematicae Universitatis Carolinae\/} {\bf 25} (1984)
667-678.

\bibitem{JCH} Jia, Z., Cai, X., Han, D. Comparison of several fast algorithms for projection
onto an ellipsoid. {\it Journal of Computational and Applied Mathematics\/} {\bf 319} (2017) 320–337.

\bibitem{Kac} Kaczmarz, S. Angenaherte Aufl\"osung von Systemen
linearer Gleichungen. {\it Bulletin de l'Acad\'emie
Polonaise des Sciences et Lettres\/} {\bf A-35} (1937)
355-357.

\bibitem{KYF} Kanzow, C., Yamashita, N., Fukushima, M.
Levenberg-Marquardt methods with strong local convergence properties
for solving nonlinear equations with convex constraints. {\it Journal of
Computational and
Applied Mathematics\/} {\bf 172} (2004) 375-397.

\bibitem{Mou} Moudafi, A. A note on the split common fixed-point problem 
for quasi nonexpansive operators. {\it Nonlinear Analysis\/} {\bf 74} (2011) 4083-4087.

\bibitem{Ouy} Ouyang, H.
Finite convergence of locally proper circumcentered methods. arXiv: 2011.13521 [math] (2020).

\bibitem{Pie} Pierra, G.
Decomposition through formalization in a product space.
{\it Mathematical Programming\/} {\bf 28} (1983) 96--115.

\bibitem{Rob} Robinson, S.M. Stability theory for systems of inequalities, Part II:
Differentiable nonlinear systems. {\it SIAM Journal on Numerical Analysis\/} 
{\bf 13} (1976) 497-513.

\bibitem{RoW} Rockafellar, R.T., Wets, R.J-B. {\it Variational Analysis.\/}
Springer, Berlin (1998).

\bibitem{YLY} Yao, Y., Liou, Y.-C., Yao, J.-C. Split common fixed point problem
for two quasi-pseudo contractive operators and its algorithm construction.
{\it Fixed Pont Theory and Applications\/} (2015) Article no. 127.

\bibitem{ZhH} Zhao, J., He, S. Alternating Mann iterative algorithms for the split common fixed-point problem of
quasi-nonexpansive maps. {\it Fixed Point Theory and Applications\/} (2013) Article no. 288.

\end{thebibliography}
\end{document}